\documentclass{amsart}

\usepackage[T1]{fontenc}
\usepackage{amsmath,amssymb,amsthm,mathtools}
\usepackage[margin=1in]{geometry}
\usepackage[colorlinks=true,linkcolor=blue,citecolor=blue,urlcolor=blue]{hyperref}

\usepackage[
    style=numeric,
    sorting=nyt,
    maxalphanames=5,
    maxnames = 5,
    url=false,
    doi=false,
    eprint=false
]{biblatex}
\title{Almost-linear Zarankiewicz bounds in 1-semi-equational theories}

\author{Hongyi Gou}
\address{Department of Mathematics, National University of Singapore, Singapore}
\email{e0708226@u.nus.edu}

\author{Mostafa Mirabi}
\address{The Taft School, Watertown, CT 06795, USA, and Wesleyan
University, Middletown, CT 06459, USA}
\email{mmirabi@wesleyan.edu}
\urladdr{https://sites.google.com/site/mostafamirabi}

\author{Mihir Mittal}
\address{Department of Mathematics, National University of Singapore, Singapore}
\email{mihirmittal24@u.nus.edu}

\author{Chieu-Minh Tran}
\address{Department of Mathematics, National University of Singapore, Singapore}
\email{trancm@nus.edu.sg}

\author{Zhenyu Yang}
\address{School of Statistics and Data Science, Nankai University, Tianjin, China}
\email{2110314@mail.nankai.edu.cn}

\keywords{semi-equational theories, $(k,1)$-semi-equations,
Zarankiewicz problem, laminar set systems, extremal hypergraphs,
almost-linear bounds}

\subjclass[2020]{Primary 03C45; Secondary 05C35, 05C65, 03C98}

\theoremstyle{plain}
\newtheorem{thm}{Theorem}[section]
\newtheorem{lem}[thm]{Lemma}
\newtheorem{prop}[thm]{Proposition}
\newtheorem{cor}[thm]{Corollary}
\newtheorem{mainthm}{Theorem}

\theoremstyle{definition}
\newtheorem{defn}[thm]{Definition}

\theoremstyle{remark}
\newtheorem{rem}[thm]{Remark}

\numberwithin{equation}{section}

\newcommand{\ind}{\mathbf 1}

\begin{document}

\begin{abstract}
We study multipartite hypergraphs definable in \(1\)-semi-equational
theories and prove almost-linear Zarankiewicz bounds in every fixed
arity \(r\geq2\). If \(T\) is a \(1\)-semi-equational theory, then,
for every formula \(\varphi\) and fixed \(t,r\geq2\), there is a
constant \(c\) such that
each \(K_{t,\ldots,t}\)-free \(r\)-partite hypergraph defined by
\(\varphi\) on \(n\) vertices has
\(O_{T,\varphi,t,r}(n^{r-1}(1+\log(1+n))^c)\) edges. Put
\(\alpha_k=\min\{k-1,2\}\). In the bipartite case, a Boolean
combination of \(m\)
\((k,1)\)-semi-equations has
\(O_{k,t,m}(n(1+\log(1+n))^{(m-1)\alpha_k})\) edges whenever it is
\(K_{t,t}\)-free. In particular, a relation defined by one
\((k,1)\)-semi-equation or its negation has a linear bound. The proofs
combine incidence estimates for indexed set systems with low-crossing
orderings of finite \(k\)-wise laminar families. Consequently, no $1$-semi-equational theory locally trace-defines an infinite domain. 
\end{abstract}

\maketitle

\section{Introduction}

For integers \(n_1,n_2\geq 1\) and \(t\geq 2\), the Zarankiewicz problem
asks for the maximum number of edges in a \(K_{t,t}\)-free bipartite graph
\(G=(V_1,V_2;E)\) with \(|V_1|=n_1\) and \(|V_2|=n_2\). Writing
\(n=n_1+n_2\), the theorem of K\H{o}v\'ari, S\'os, and
Tur\'an~\cite{kovari1954problem} gives
\(|E|=O_t(n^{2-1/t})\); when \(t=2\), this exponent is
sharp~\cite{reiman1958problem}. Additional structure on the edge relation
can yield stronger estimates. For example, the point-line incidence
relation in \(\mathbb R^2\) is \(K_{2,2}\)-free, and the
Szemer\'edi--Trotter theorem~\cite{szemeredi1983extremal} gives the sharp
estimate \(O(n^{4/3})\), improving the general \(O(n^{3/2})\) bound.

Model theory and incidence geometry provide several frameworks for
obtaining improved Zarankiewicz bounds. Fox, Pach, Sheffer, Suk, and
Zahl~\cite{fox2017semi} established power-saving bounds for
semialgebraic graphs, and Do~\cite{Do2018} obtained corresponding
higher-arity estimates. Chernikov, Galvin, and Starchenko
\cite{chernikov2020cutting} extended the binary estimates to relations
definable in distal structures, while Tong~\cite{Tong2026} obtained
higher-arity bounds from distal regularity.

Basit, Chernikov, Starchenko, Tao, and Tran \cite{basit2021zarankiewicz} proved almost-linear bounds for semilinear hypergraphs. 
Several subsequent approaches have obtained incidence bounds in geometric settings under the \(K_{t,t}\)-free assumption; the ones most closely related to the present work include the geometric incidence estimates of Chan and Har-Peled~\cite{ChanHarPeled2023}, who gave an \(O(n(\log n/\log\log n)^{d-1})\) bound for axis-parallel boxes in \(\mathbb{R}^d\); the \(\varepsilon\)-\(t\)-net method of Keller and Smorodinsky~\cite{KellerSmorodinsky2024}, yielding a simple proof of the Zarankiewicz bound for semialgebraic graphs; the factorization-norm method of Tomon~\cite{TomonFactorization}, which shows that bounded \(\gamma_2\)-norm implies bounded degree in \(K_{t,t}\)-free graphs; and the structural method with geometric applications developed by Hunter, Milojevi\'c, Sudakov, and Tomon~\cite{HunterMilojevicSudakovTomon2025}, obtaining \(C_4\)-free subgraph bounds with applications to incidence problems. Higher-arity geometric intersection problems were studied by Chan, Keller, and Smorodinsky~\cite{ChanKellerSmorodinsky2025} and, for axis-parallel boxes, by Chao, Dong, Liu, Shu, and Wang~\cite{ChaoDongLiuShuWang}, who established a dichotomy for hypergraph Zarankiewicz problems.

In the model-theoretic direction, Eleftheriou and
Papadopoulos~\cite{EleftheriouPapadopoulos2025} obtained linear
Zarankiewicz bounds in several global settings. Their hypotheses concern
a fixed globally \(K_{t,\ldots,t}\)-free definable relation, whereas the
results of the present paper require only that the finite induced
hypergraph under consideration be \(K_{t,\ldots,t}\)-free.

We study the semi-equational framework introduced by Chernikov and Mennen~\cite{ChernikovMennen}. Let \(\mathcal F\) be an indexed family of subsets of a set \(X\). The family is \emph{laminar} if any two of its members are disjoint or comparable under inclusion. For \(k\geq2\), it is \emph{\(k\)-wise laminar} if, whenever \(k\) distinct indices determine sets with nonempty common intersection, two of those sets are comparable. Thus \(2\)-wise laminarity is ordinary laminarity. Equal members are comparable.

Let \(M\) be a structure. Following~\cite{ChernikovMennen}, a
partitioned formula \(\phi(x;y)\) is a
\emph{\((k,1)\)-semi-equation in \(M\)} if, whenever
\(b_1,\ldots,b_k\in M^{|y|}\) and
\(M\models\exists x\bigwedge_{j=1}^k\phi(x;b_j)\), there are distinct
\(i,j\in\{1,\ldots,k\}\) such that
\(M\models\forall x(\phi(x;b_i)\rightarrow\phi(x;b_j))\).
Equivalently, the indexed family
\((\phi(M^{|x|};b))_{b\in M^{|y|}}\) is \(k\)-wise laminar.
A theory is \((k,1)\)-semi-equational if every partitioned formula is
equivalent, modulo the theory, to a finite Boolean combination of
\((k,1)\)-semi-equations. It is \emph{\(1\)-semi-equational} if every
partitioned formula is equivalent to a finite Boolean combination of
\((k,1)\)-semi-equations, where \(k\) may depend on the formula. This is
the terminology of Definition~2.13 in~\cite{ChernikovMennen}. This notion extends a familiar form of model-theoretic linearity.
By~\cite[Proposition~2.19]{ChernikovMennen}, a stable theory is
\(1\)-semi\-equational if and only if it is \(1\)-based. Thus
\(1\)-semi\-equationality may be viewed as a form of \(1\)-basedness
beyond stability. Unstable examples include arbitrary unary expansions
of linear orders and several classes of ordered abelian groups; see
\cite[Sections~3.2 and~3.4]{ChernikovMennen}. They proved
that every \((2,1)\)-semi-equation is the
intersection of two, not necessarily definable, basic relations, where a
basic relation has the form \(f(a)<g(b)\) for maps \(f,g\) into a linear
order; see Proposition~2.23 in~\cite{ChernikovMennen}. They asked whether every
\((k,1)\)-semi-equation with \(k\geq3\) is a Boolean combination of basic
relations; this is Problem~2.25 in~\cite{ChernikovMennen}, and it remains open.
We instead prove the almost-linear Zarankiewicz bounds directly from
\(k\)-wise laminarity, without using a decomposition into basic
relations. The bipartite case answers the Zarankiewicz question in
Problem~6.2 of~\cite{ChernikovStarchenkoOneBased}, while the higher-arity
theorem addresses the Zarankiewicz aspect of Problem~6.1 in the same
paper. It does not address the
distal-cell-decomposition question in Problem~6.2 or the basic-relation
decomposition problem above. On the other hand,
natural valuative comparisons need not be \((k,1)\)-semi-equations for
any finite \(k\). Laminar families nevertheless arise after restriction
to constant-valuation components and coordinate projection. See
Example~3.1 and Proposition~3.3 of
\cite{gou2026zarankiewiczsboundsvaluedvector}.
This suggests that a more general semi-equational framework may have to
incorporate such local or projection-laminar structure.

The first version of this paper proved the signed extension estimate by
a direct recursion from \(k\)-wise to \((k+1)\)-wise laminar families;
this argument is retained in Sections~\ref{sect:positive-induction}
and~\ref{sect:negative-induction}. The linear-size input used in the
sharpening below follows from Knill's theorem on families of locally
bounded width~\cite{KnillLocallyBoundedWidth} and is also independently proved directly
for \(k\)-wise laminar families by Tomon in
\cite[Theorem~1.2]{TomonCrossFree}. After the first version appeared,
Tomon brought this bound to our attention and separately suggested
combining its dual consequence with the low-crossing path theorem in the
formulation of Bonnet, Duron, Sylvester, and
Zamaraev~\cite{BonnetDuronSylvesterZamaraev}, which is based on Welzl's
method~\cite{Welzl1988}, and a dyadic interval decomposition. This gives
the sharper exponent below and is proved in
Section~\ref{sect:sharpening}.

We first state the absolute higher-arity result, followed by the explicit graph bounds. Throughout the paper, \(O_\eta(\cdot)\) denotes a quantity bounded in absolute value by a constant depending only on the parameters \(\eta\). In the sections below, constants denoted by \(C_\eta\) have the same dependence and may vary between occurrences.

For finite sets \(V_1,\ldots,V_r\), a relation
\(E\subseteq V_1\times\cdots\times V_r\) is
\emph{\(K_{t,\ldots,t}\)-free} if there are no subsets
\(A_j\subseteq V_j\), all of cardinality \(t\), such that
\(A_1\times\cdots\times A_r\subseteq E\). For \(r=2\), we write
\(K_{t,t}\)-free.

The tuples \(x_1,\ldots,x_r\) in Theorem~\ref{thm:absolute-main}
correspond to the \(r\) vertex classes. The tuple \(z\) is an auxiliary
parameter and does not determine an additional part of the hypergraph.
In the bipartite statements, fixed auxiliary parameters are suppressed
and the two vertex classes are denoted by \(x\) and \(y\). In the proof
of Theorem~\ref{thm:absolute-main}, semi-equationality is applied to the
partition \((x_1,\ldots,x_{r-1});(x_r,z)\).

\begin{mainthm}\label{thm:absolute-main}
Let \(T\) be a \(1\)-semi-equational theory and let \(M\models T\).
Fix \(t,r\geq2\) and a formula \(\varphi(x_1,\ldots,x_r;z)\). There is
\(c\geq0\), depending only on
\(T,\varphi,t,r\), such that the following holds uniformly for all
\(d\in M^{|z|}\). If
\(V_j\subseteq M^{|x_j|}\) are finite, \(n=\sum_{j=1}^r|V_j|\), and
\[
 E=\left\{(a_1,\ldots,a_r)\in
 V_1\times\cdots\times V_r:
 M\models\varphi(a_1,\ldots,a_r;d)\right\}
\]
is \(K_{t,\ldots,t}\)-free, then
\[
 |E|=O_{T,\varphi,t,r}\!\left(
 n^{r-1}\bigl(1+\log(1+n)\bigr)^c\right).
\]
\end{mainthm}

Put \(\alpha_k=\min\{k-1,2\}\). The constant-base conclusion of
Theorem~\ref{thm:boolean-combinations} gives the following bipartite
estimate.

\begin{cor}\label{cor:bipartite-main}
Let \(M\) be a structure, let \(k,t\geq2\), and let
\(\phi_1(x;y),\ldots,\phi_m(x;y)\) be \((k,1)\)-semi-equations in
\(M\), where \(m\geq1\). Let
\(\Psi\colon\{0,1\}^m\to\{0,1\}\). For finite sets
\(V\subseteq M^{|x|}\) and \(W\subseteq M^{|y|}\), put
\(n=|V|+|W|\) and define
\[
 E=\left\{(a,b)\in V\times W:
 \Psi\bigl(\ind_{M\models\phi_1(a;b)},\ldots,
 \ind_{M\models\phi_m(a;b)}\bigr)=1\right\}.
\]
If \(E\) is \(K_{t,t}\)-free, then
\begin{equation}\label{eq:bipartite-boolean}
 |E|=O_{k,t,m}\!\left(
 n(1+\log(1+n))^{(m-1)\alpha_k}\right).
\end{equation}
In particular, if \(m=1\), then \(|E|=O_{k,t}(n)\).
\end{cor}

Informally, a structure locally trace-defines another structure if every finite trace of a relation definable in the latter can be realized as the trace of a definable relation in the former.

\begin{cor}\label{cor:trace-domains}
Let \(T\) be a \(1\)-semi-equational theory. Then \(T\) has
almost-linear Zarankiewicz bounds for its definable bipartite graphs and
does not locally trace-define an infinite domain.
\end{cor}

\begin{proof}
Fix a definable bipartite relation. By the definition of a
\(1\)-semi-equational theory, it is equivalent to a finite Boolean
combination of \((k_q,1)\)-semi-equations. If the combination has no
literals, adjoin a tautological \((2,1)\)-semi-equation and let the
Boolean function ignore it. Taking \(k=\max_q k_q\), every
one of these formulas is a \((k,1)\)-semi-equation.
Corollary~\ref{cor:bipartite-main} gives a bound of the form
\(O(n(1+\log(1+n))^c)\) for some \(c\) depending on the relation. For
every \(\varepsilon>0\),
\(n(1+\log(1+n))^c=O_{\varepsilon,c}(n^{1+\varepsilon})\).
Walsberg proves that this \(O(n^{1+\varepsilon})\) property is preserved
under local trace definability and that it fails in every infinite
domain~\cite[Lemma~3.4 and Proposition~3.8]{WalsbergTraceII}. The
conclusion follows.
\end{proof}

\begin{rem}\label{rem:one-based-comparison}
By \cite[Theorem~4.11]{ChernikovStarchenkoOneBased}, every relation
definable in a stable one-based theory has a linear Zarankiewicz bound.
Together with the results of Walsberg cited above, this independently
implies that a stable one-based theory does not locally trace-define an
infinite domain.
\end{rem}

The model-theoretic input enters in two places: through the finite
Boolean decomposition into \((k,1)\)-semiequations and, in higher
arity, through the uniform definability of intersections of fibers.
Section~\ref{sect:setup} establishes the laminar and co-laminar base
cases. Sections~\ref{sect:positive-induction} and
\ref{sect:negative-induction} develop the recursive extension method,
while Section~\ref{sect:sharpening} obtains the sharper estimate from a
low-crossing ordering. Section~\ref{sect:signed} treats signed families,
conjunctions, and Boolean combinations. Finally,
Section~\ref{sect:higher-arity} applies these combinatorial estimates to
definable hypergraphs and proves Theorem~\ref{thm:absolute-main} by induction on \(r\).

\subsection*{Acknowledgments}
We thank Erik Walsberg and Artem Chernikov for helpful discussions
concerning this work, and Mervyn Tong for valuable comments on an earlier
draft. We are grateful to Istv\'an Tomon for bringing the linear-size
bound for simple \(k\)-wise laminar families to our attention and for
suggesting its combination with a low-crossing ordering and a dyadic
interval decomposition, which leads to the improved bounds in
Corollary~\ref{cor:bipartite-main}. The authors used an AI-assisted tool
for language editing and consistency checks. All mathematical results
and arguments were independently proved by the authors. The authors take
full responsibility for the mathematical content and the final text.

\subsection*{Author contributions}
Hongyi Gou, Mihir Mittal, Chieu-Minh Tran, and Zhenyu Yang contributed
equally to the general combinatorial framework and the proof 
of the higher-arity result. Mostafa Mirabi contributed the direct signed-base
argument and the resulting improved bounds for constant base systems.

\section{Laminar and co-laminar base case}\label{sect:setup}

Fix integers \(t\geq2\) and \(B\geq0\), a finite set \(X\), and a
finite index set \(I\). An \emph{indexed set system on \(X\)} is a family
\(\mathcal F=(F_i)_{i\in I}\) with \(F_i\subseteq X\). Its weight is
\(w(\mathcal F)=\sum_{i\in I}|F_i|\).
If \(\mathcal A\subseteq2^X\), we also write \(\mathcal A\) for the
indexed system \((A)_{A\in\mathcal A}\).
If \(J\subseteq I\) and \(Y\subseteq X\), the \emph{induced restriction}
of \(\mathcal F\) to \((J,Y)\) is
\(\mathcal F\mathbin{|}(J,Y)=(F_i\cap Y)_{i\in J}\).
A system \(\mathcal G=(G_i)_{i\in I}\) is a \emph{subsystem} of \(\mathcal F\) if \(G_i\subseteq F_i\) for every \(i\in I\).

We say that \(\mathcal F\) has \emph{\(t\)-fold intersections bounded by
\(B\)} if \(|F_{i_1}\cap\cdots\cap F_{i_t}|\leq B\) for all distinct
\(i_1,\ldots,i_t\in I\). This property is preserved by
induced restrictions and by passage to a subsystem. Define
\[
 \Delta_B(Y,J)=|Y|+(B+1)|J|,
 \qquad N_B=\Delta_B(X,I),
 \qquad \ell_B=1+\log(1+N_B).
\]
The coefficient \(B+1\) ensures that \(\Delta_B\) decreases strictly
along every nonempty proper recursive child.

Let \(\mathcal C\) be a class of indexed set systems closed under
induced restrictions. From this point through Section~\ref{sect:signed},
assume that there is a fixed number
\(\Gamma\geq1\) such that, for every \(Y\subseteq X\), every
\(J\subseteq I\), and every system \(\mathcal G\in\mathcal C\) on
\((J,Y)\) whose \(t\)-fold intersections are bounded by \(B\),
\begin{equation}\label{eq:base-estimate}
 w(\mathcal G)\leq\Gamma\Delta_B(Y,J).
\end{equation}

\begin{lem}\label{lem:constant-class}
The class of constant systems, namely systems of the form
\((Y)_{i\in J}\), satisfies~\eqref{eq:base-estimate} with
\(\Gamma=t-1\).
\end{lem}

\begin{proof}
Suppose that \((Y)_{i\in J}\) has \(t\)-fold intersections bounded by
\(B\). If \(|J|<t\), then
\(|J||Y|\leq(t-1)|Y|\leq(t-1)\Delta_B(Y,J)\).
If \(|J|\geq t\), then \(|Y|\leq B\), and hence
\(|J||Y|\leq B|J|\leq\Delta_B(Y,J)\).
The class is closed under induced restrictions, so the assertion follows.
\end{proof}

An indexed family \(\mathcal D=(D_i)_{i\in I}\) is
\emph{co-laminar} if \((X\setminus D_i)_{i\in I}\) is laminar.
Empty sets and repeated members are allowed in both definitions.

\begin{prop}\label{prop:base-extension}
Let \(\mathcal F'=(F'_i)_{i\in I}\in\mathcal C\), and let
\(\mathcal D=(D_i)_{i\in I}\) be laminar or co-laminar. Suppose that
\(F_i=F'_i\cap D_i\) for \(i\in I\), defining a system
\(\mathcal F=(F_i)_{i\in I}\) whose \(t\)-fold
intersections are bounded by \(B\). Then
\[
 w(\mathcal F)\leq C_t\Gamma N_B\ell_B.
\]
More precisely, for every \(J\subseteq I\) and \(Y\subseteq X\),
\[
 w\bigl(\mathcal F\mathbin{|}(J,Y)\bigr)
 \leq C_t\Gamma\Delta_B(Y,J)\ell_B.
\]
\end{prop}

The interval representation and balanced interval-tree decomposition in
the following proof use the laminar-family counting construction from
\cite[Section~4, Lemmas~4.1 and~4.2]{gou2026zarankiewiczsboundsvaluedvector}.

\begin{proof}
If \(X=\varnothing\), then \(w(\mathcal F)=0\). Assume
\(X\neq\varnothing\), and set \(h=1+\lceil\log_2|X|\rceil\).

First suppose that \(\mathcal D\) is laminar. Pass to the family of
distinct nonempty sets occurring among the \(D_i\). Its strict-inclusion
relation is a forest: distinct maximal members are disjoint, and each
nonmaximal member has a unique minimal strict superset in the family. For
each member \(A\), its children are pairwise disjoint. Order the children
arbitrarily and recursively list first the points of
\(A\setminus\bigcup\{C:C\text{ is a child of }A\}\), and then the point
lists of its children. Concatenate the lists of the maximal members and
then the points outside their union. The resulting linear order on \(X\)
makes every member of the family an interval. Repeated members use the
same interval, and empty members require no interval. This finite order is
used only for the incidence decomposition below; it need not be definable
or compatible with any order on the index set.

Form the balanced binary interval tree by recursively splitting every
non\-singleton interval into two consecutive intervals whose
cardinalities differ by at most one. For every interval \(Q\subseteq
X\), let \(\mathcal K(Q)\) be the set of maximal tree nodes whose
associated intervals \(C_u\) lie in \(Q\). The intervals
\((C_u)_{u\in\mathcal K(Q)}\) are disjoint and cover \(Q\). At each
depth, only the two nodes meeting the boundary of \(Q\) can have
descendants in \(\mathcal K(Q)\), and therefore
\(|\mathcal K(Q)|\leq2h\). For each \(i\in I\), put
\(\mathcal K_i=\mathcal K(D_i)\), and, for every tree node \(u\), put
\(I_u=\{i\in I:u\in\mathcal K_i\}\).
Every incidence of \(\mathcal F\) belongs to a unique block
\(I_u\times C_u\). Moreover, \(C_u\subseteq D_i\) for \(i\in I_u\), so
\[
 \mathcal F\mathbin{|}(I_u,C_u)
 =\mathcal F'\mathbin{|}(I_u,C_u).
\]
The left-hand side has \(t\)-fold intersections bounded by \(B\), while
the right-hand side belongs to \(\mathcal C\). By~\eqref{eq:base-estimate},
\(w(\mathcal F)\leq\Gamma\sum_u\Delta_B(C_u,I_u)\). At each depth the
\(C_u\) are disjoint, so \(\sum_u|C_u|\leq h|X|\).
Moreover, \(\sum_u|I_u|\leq2h|I|\). Thus
\(w(\mathcal F)\leq C\Gamma hN_B\).

Now suppose that \(\mathcal D\) is co-laminar, and put
\(K_i=X\setminus D_i\). Apply the preceding ordering to the laminar family
\((K_i)_{i\in I}\). Each \(D_i\) is the union of at most two intervals,
so it has a canonical decomposition into at most \(4h\) tree nodes.
Applying~\eqref{eq:base-estimate} to the resulting blocks gives
\(w(\mathcal F)\leq C\Gamma hN_B\).
Since \(h=O(\ell_B)\), the asserted estimate follows in both cases. For
an induced restriction to \((J,Y)\), apply the construction to the ground
set \(Y\) and the restricted family \((D_i\cap Y)_{i\in J}\). Its interval
tree has height at most
\(1+\lceil\log_2(1+|Y|)\rceil=O(\ell_B)\), and the two preceding sums are
bounded by that height times \(|Y|\) and \(|J|\), respectively. Thus
\(w(\mathcal F\mathbin{|}(J,Y))
\leq C_t\Gamma\Delta_B(Y,J)\ell_B\), as required.
\end{proof}

\section[Extension by k-wise laminar families]
{Extension by \texorpdfstring{\(k\)}{k}-wise laminar families}
\label{sect:positive-induction}

For \(k\geq2\), an indexed family \(\mathcal D=(D_i)_{i\in I}\) is
\(k\)-wise laminar if, for all distinct \(i_1,\ldots,i_k\in I\),
\(D_{i_1}\cap\cdots\cap D_{i_k}\neq\varnothing\) implies that two of the
selected sets are comparable under inclusion. Restrictions to a subset of
the ground set preserve this property.

For a number \(\Lambda\geq1\), let \(\mathsf L_k(\Lambda)\) denote the
following extension property. There is a constant \(A_{k,t}\) such that,
for every \(J\subseteq I\), \(Y\subseteq X\), every base system
\(\mathcal F'\in\mathcal C\) on \((J,Y)\), and every \(k\)-wise laminar
family \(\mathcal L=(L_i)_{i\in J}\) on \(Y\), if
\((F'_i\cap L_i)_{i\in J}\) has \(t\)-fold intersections bounded by \(B\),
then its weight is at most
\begin{equation}\label{eq:Lk-property}
 A_{k,t}\Gamma\Delta_B(Y,J)\Lambda.
\end{equation}
Proposition~\ref{prop:base-extension} gives \(\mathsf L_2(\ell_B)\).

We prove that passage from \(k\) to \(k+1\) introduces one additional
factor \(\ell_B\). Let \(\mathcal F'=(F'_i)_{i\in I}\in\mathcal C\), let
\(\mathcal D=(D_i)_{i\in I}\) be \((k+1)\)-wise laminar, and suppose that
\(F_i=F'_i\cap D_i\) defines a system \(\mathcal F=(F_i)_{i\in I}\) whose
\(t\)-fold
intersections are bounded by \(B\).

For \(J\subseteq I\) and \(Y\subseteq X\), write
\(W(J,Y)=\sum_{i\in J}|F_i\cap Y|\) and
\(p(J,Y)=\Delta_B(Y,J)\).
For a pair \(v=(J,Y)\), we also write \(W(v)=W(J,Y)\) and
\(p(v)=p(J,Y)\).
Whenever a pair \((J,Y)\) occurs below, indices satisfying \(D_i\cap Y=\varnothing\) are deleted. This leaves \(W(J,Y)\) unchanged and does not increase \(p(J,Y)\).
A pair is \emph{terminal} if its index set or ground set is empty after
this deletion.

Fix a nonempty pair \((J,Y)\) after this deletion. Choose \(a\in J\) such
that \(U=D_a\cap Y\) is maximal under inclusion among the sets
\(D_i\cap Y\), \(i\in J\). Define
\[
 \begin{aligned}
 J_{\mathrm{in}}&=\{i\in J:D_i\cap Y\subsetneq U\},\\
 J_{\mathrm{out}}&=\{i\in J:(D_i\cap Y)\cap U=\varnothing\},\\
 J_{\mathrm{eq}}&=\{i\in J:D_i\cap Y=U\},\\
 J_{\times}&=J\setminus
 (J_{\mathrm{in}}\cup J_{\mathrm{out}}\cup J_{\mathrm{eq}}).
 \end{aligned}
\]
These four classes partition \(J\). If \(i\in J_{\times}\), then
\(D_i\cap Y\) meets both \(U\) and \(Y\setminus U\); moreover, maximality
of \(U\) excludes \(U\subsetneq D_i\cap Y\). Thus every member of
\(J_{\times}\) crosses \(U\). The two recursive pairs are
\((J_{\mathrm{in}},U)\) and
\((J_{\mathrm{out}}\cup J_{\times},Y\setminus U)\), after indices with
empty restrictions are deleted. Equality indices are absent from the
outside child because their restrictions to \(Y\setminus U\) are empty.
They satisfy
\begin{equation}\label{eq:child-mass}
 p(J_{\mathrm{in}},U)
 +p(J_{\mathrm{out}}\cup J_{\times},Y\setminus U)
 \leq p(J,Y).
\end{equation}

We call \((J_{\mathrm{in}},U)\) the \emph{inside child} and
\((J_{\mathrm{out}}\cup J_{\times},Y\setminus U)\) the
\emph{outside child}. These terms refer only to the side of the
partition of \(Y\) on which the recursive pair is supported.

\begin{lem}\label{lem:one-step-reduction}
For \(i\in J_{\mathrm{eq}}\cup J_{\times}\), put
\(E_i=(D_i\cap Y)\cap U\). Then
\((E_i)_{i\in J_{\mathrm{eq}}\cup J_{\times}}\) is \(k\)-wise
laminar.
\end{lem}

\begin{proof}
Let \(i_1,\ldots,i_k\in J_{\mathrm{eq}}\cup J_{\times}\) be distinct
and suppose that the \(E_{i_j}\) have nonempty common intersection. If
some \(i_j\) belongs to \(J_{\mathrm{eq}}\), then \(E_{i_j}=U\) contains
every selected set. We may therefore assume that every \(i_j\) belongs
to \(J_{\times}\).

Each \(D_{i_j}\cap Y\) crosses \(U\). A point in
\((D_{i_j}\cap Y)\setminus U\) shows that
\(D_{i_j}\nsubseteq D_a\), while \(D_a\subseteq D_{i_j}\) would imply
\(U\subseteq D_{i_j}\cap Y\), contrary to maximality and
\(i_j\notin J_{\mathrm{eq}}\). Hence \(D_{i_j}\) and \(D_a\) are
incomparable.

The indices \(i_1,\ldots,i_k,a\) are distinct, and their sets have nonempty common intersection. Since \(\mathcal D\) is \((k+1)\)-wise laminar, two of these sets are comparable. No selected set is comparable with \(D_a\), so two of \(D_{i_1},\ldots,D_{i_k}\) are comparable. Intersecting with \(Y\cap U\) preserves the inclusion.
\end{proof}

Assume \(\mathsf L_k(\Lambda)\). The system
\(\mathcal E=(F'_i\cap E_i)_{i\in J_{\mathrm{eq}}\cup J_{\times}}\)
is an induced
restriction of \(\mathcal F\), so its \(t\)-fold
intersections are bounded by \(B\). Lemma~\ref{lem:one-step-reduction}
and~\eqref{eq:Lk-property} therefore give
\(w(\mathcal E)\leq A_{k,t}\Gamma p(J,Y)\Lambda\).
Consequently,
\begin{equation}\label{eq:positive-recurrence}
 \begin{split}
 W(J,Y)\leq{}&W(J_{\mathrm{in}},U)
 +W(J_{\mathrm{out}}\cup J_{\times},Y\setminus U)\\
 &+A_{k,t}\Gamma p(J,Y)\Lambda.
 \end{split}
\end{equation}

A child of \((J,Y)\) is \emph{heavy} if its parameter exceeds
\(3p(J,Y)/4\). By~\eqref{eq:child-mass}, at most one child is heavy. A
pair is \emph{balanced} if neither child is heavy. A \emph{heavy
branch} is obtained by repeatedly passing to the unique heavy child.
A step on such a branch is an \emph{inside step} or an \emph{outside
step} according to which child is followed. At a balanced pair,
~\eqref{eq:positive-recurrence} has total recursive parameter at most
\(p(J,Y)\), and every recursive parameter is at most \(3p(J,Y)/4\).
Starting from \(v_0\), a heavy branch is stopped at the first pair
\(v_s\) that is terminal, is balanced, or satisfies
\(p(v_s)\leq 3p(v_0)/4\). Such a pair exists because \(p\) is a
positive integer and decreases strictly whenever the recursion passes
to a nonempty child.

The next two lemmas use only that each successive pair is an inside or
outside child of its predecessor; their proofs do not use heaviness.
Consider a heavy branch
\[
 v_0=(J_0,Y_0)\longrightarrow\cdots\longrightarrow
 v_s=(J_s,Y_s).
\]
At step \(q<s\), write \(U_q=D_{a_q}\cap Y_q\). Let \(\mathcal O\)
be the set of outside steps. Lemma~\ref{lem:surviving-indices} controls
the incidences of indices in \(J_s\) on \(Y_0\setminus Y_s\), while
Lemma~\ref{lem:deleted-indices} controls the incidences on \(Y_s\) of
indices that leave the branch before its terminal pair.

\begin{lem}\label{lem:surviving-indices}
Put \(I_1=J_s\), \(X_1=Y_s\), and \(X_2=Y_0\setminus Y_s\).
For \(i\in I_1\), define
\[
 P_i=\bigcup_{\substack{q\in\mathcal O\\i\in J_{\times}^q}}
 (D_i\cap U_q).
\]
Then \(P_i=D_i\cap X_2\), and consequently
\(F_i\cap X_2=F'_i\cap P_i\). Moreover, the nonempty members of
\((P_i)_{i\in I_1}\) form a \(k\)-wise laminar family on \(X_2\).
\end{lem}

\begin{proof}
For \(q<s\), put \(Z_q=Y_q\setminus Y_{q+1}\). These sets are
pairwise disjoint and have union \(X_2\). Fix \(i\in I_1\). Since
\(i\) survives to \(J_s\), it survives every step of the branch.

If \(q\) is an inside step, then \(Y_{q+1}=U_q\), and the survival
of \(i\) implies \(D_i\cap Y_q\subsetneq U_q\). Hence
\(D_i\cap Z_q=\varnothing\). If \(q\) is an outside step, then
\(Z_q=U_q\). Since the outside child is indexed by
\(J_{\mathrm{out}}^q\cup J_{\times}^q\), a surviving index belongs to
one of these two classes. If \(i\in J_{\mathrm{out}}^q\), then
\(D_i\cap Z_q=\varnothing\), whereas if \(i\in J_{\times}^q\), then
\(D_i\cap Z_q=D_i\cap U_q\). It follows that
\[
 D_i\cap X_2
 =\bigcup_{q<s}(D_i\cap Z_q)
 =\bigcup_{\substack{q\in\mathcal O\\i\in J_{\times}^q}}
   (D_i\cap U_q)
 =P_i.
\]
The identity \(F_i\cap X_2=F'_i\cap P_i\) is immediate.

Suppose that distinct \(i_1,\ldots,i_k\in I_1\) satisfy
\(P_{i_1}\cap\cdots\cap P_{i_k}\neq\varnothing\). The sets removed
at outside steps are pairwise disjoint, so the common point belongs
to a unique \(U_q\), where \(q\in\mathcal O\). Every \(i_j\) belongs
to \(J_{\times}^q\), and hence \(D_{i_j}\cap Y_q\) crosses \(U_q\).

The sets \(D_{i_1},\ldots,D_{i_k}\) are therefore incomparable with
the pivot \(D_{a_q}\). They have a common point with \(D_{a_q}\), so
\((k+1)\)-wise laminarity gives comparable sets among
\(D_{i_1},\ldots,D_{i_k}\), say
\(D_{i_\alpha}\subseteq D_{i_\gamma}\).

At every outside step \(q'\) at which \(i_\alpha\) crosses \(U_{q'}\),
the set \(D_{i_\gamma}\cap Y_{q'}\) also crosses \(U_{q'}\): it
contains points on both sides of \(U_{q'}\), and maximality prevents
it from strictly containing \(U_{q'}\). Therefore every term in
\(P_{i_\alpha}\) is contained in the corresponding term of
\(P_{i_\gamma}\), and \(P_{i_\alpha}\subseteq P_{i_\gamma}\).
\end{proof}

\begin{lem}\label{lem:deleted-indices}
With the same branch, put \(I_2=J_0\setminus J_s\) and \(X_1=Y_s\).
For \(i\in I_2\), define
\[
 R_i=
 \begin{cases}
 D_i\cap X_1,&F_i\cap X_1\neq\varnothing,\\
 \varnothing,&F_i\cap X_1=\varnothing.
 \end{cases}
\]
Then \(F_i\cap X_1=F'_i\cap R_i\), and the nonempty members of
\((R_i)_{i\in I_2}\) form a \(k\)-wise laminar family. The same
conclusion holds for all nonempty sets \(D_i\cap X_1\), \(i\in I_2\).
\end{lem}

\begin{proof}
The displayed identity follows directly from the definition of \(R_i\).
It remains to prove the laminarity assertion.

Let \(i_1,\ldots,i_k\in I_2\) be distinct, and suppose that their
nonempty \(R\)-sets have a common point. If one of them equals
\(X_1\), it contains all the others. We may therefore suppose that
every \(R_{i_j}\) is a proper subset of \(X_1\).

An index with \(D_i\cap X_1\neq\varnothing\) cannot leave the branch
at an outside step. Indeed, an index omitted from the outside child
has its restricted \(D_i\)-set contained in \(U_q\), while all later
ground sets, including \(X_1\), lie in \(Y_q\setminus U_q\). Thus
each \(i_j\) leaves at an inside step; denote this step by \(q_j\).

At step \(q_j\), the set \(D_{i_j}\cap Y_{q_j}\) meets
\(X_1\subseteq U_{q_j}\), but \(i_j\) does not enter the inside
child. It is therefore not in \(J_{\mathrm{out}}^{q_j}\). It is not in
\(J_{\mathrm{eq}}^{q_j}\) either, since equality would give
\(R_{i_j}=X_1\). Hence \(i_j\in J_{\times}^{q_j}\), so
\(D_{i_j}\cap Y_{q_j}\) crosses \(U_{q_j}\).

Let \(q=\max_jq_j\). Then \(X_1\subseteq U_q\subseteq D_{a_q}\).
Since \(R_{i_j}\subsetneq X_1\), one has
\(D_{a_q}\nsubseteq D_{i_j}\). On the other hand,
\(D_{i_j}\cap Y_{q_j}\) contains a point outside \(U_{q_j}\).
The pivot index \(a_q\) survives every earlier inside step, so
\(D_{a_q}\cap Y_{q_j}\subseteq U_{q_j}\), with equality when
\(q_j=q\). Thus \(D_{i_j}\nsubseteq D_{a_q}\).

Consequently, every \(D_{i_j}\) is incomparable with \(D_{a_q}\).
These sets have a common point in \(X_1\). Moreover, \(a_q\) is
distinct from every \(i_j\), since \(a_q=i_j\) would imply
\(R_{i_j}=X_1\). The \((k+1)\)-wise laminarity of \((D_i)\) therefore
gives comparable sets among \(D_{i_1},\ldots,D_{i_k}\). Intersecting
with \(X_1\) proves that two of the corresponding \(R\)-sets are
comparable.

The same argument applies directly to the nonempty sets
\(D_i\cap X_1\); the condition \(F_i\cap X_1\neq\varnothing\) is
used only to replace empty error fibers by \(\varnothing\).
\end{proof}

\begin{prop}\label{prop:heavy-compression}
Assume \(\mathsf L_k(\Lambda)\). Let \(v_0=(J_0,Y_0)\) have a heavy
child. Follow heavy children until the first pair \(v_s=(J_s,Y_s)\) that
is terminal, is balanced, or satisfies \(p(v_s)\leq3p(v_0)/4\).
Then \(W(v_0)\) is bounded by the weights of finitely many recursive systems and error systems such that:
\begin{enumerate}
\item the recursive systems have total parameter at most \(p(v_0)\), and each has parameter at most \(3p(v_0)/4\);
\item the total weight of the error systems is at most
\(C_{k,t}\Gamma p(v_0)\Lambda\).
\end{enumerate}
\end{prop}

\begin{proof}
Put \(p_0=p(v_0)\). The stopping time is finite because \(p\) is a
nonnegative integer and decreases strictly whenever the recursion
passes to a nonempty proper child.

Let \(I_1=J_s\), \(I_2=J_0\setminus J_s\),
\(X_1=Y_s\), and \(X_2=Y_0\setminus Y_s\). Define the diagonal
systems
\[
 \mathcal F_1=\mathcal F\mathbin{|}(I_1,X_1),
 \qquad
 \mathcal F_2=\mathcal F\mathbin{|}(I_2,X_2).
\]
Since \(I_1\sqcup I_2=J_0\) and \(X_1\sqcup X_2=Y_0\), their
parameters satisfy
\(p(I_1,X_1)+p(I_2,X_2)=p_0\), with
\(p(I_1,X_1)=p(v_s)\).

Let \(P_i\), \(i\in I_1\), and \(R_i\), \(i\in I_2\), be given by
Lemmas~\ref{lem:surviving-indices} and
\ref{lem:deleted-indices}. Define
\[
 \begin{aligned}
 \mathcal E_1
 &=(F'_i\cap P_i)_{i\in I_1}
   =\mathcal F\mathbin{|}(I_1,X_2),\\
 \mathcal E_2
 &=(F'_i\cap R_i)_{i\in I_2}
   =\mathcal F\mathbin{|}(I_2,X_1).
 \end{aligned}
\]
These equalities follow from the exact identity
\(P_i=D_i\cap X_2\) in Lemma~\ref{lem:surviving-indices} and the
identity in Lemma~\ref{lem:deleted-indices}. Thus
\(\mathcal E_1\) and \(\mathcal E_2\) are induced restrictions of
\(\mathcal F\), so their \(t\)-fold intersections are bounded by
\(B\). Their constraint families are \(k\)-wise laminar. Hence
\(\mathsf L_k(\Lambda)\) gives
\[
 \begin{aligned}
 w(\mathcal E_1)+w(\mathcal E_2)
 &\leq C_{k,t}\Gamma
 \bigl(p(I_1,X_2)+p(I_2,X_1)\bigr)\Lambda\\
 &=C_{k,t}\Gamma p_0\Lambda.
 \end{aligned}
\]
The final equality follows from the same partitions of the ground
set and index set.

The four systems \(\mathcal F_1,\mathcal F_2,\mathcal E_1\), and
\(\mathcal E_2\) partition all incidences of \(\mathcal F\) on
\((J_0,Y_0)\). It remains to determine which diagonal systems are
retained recursively.

Suppose first that \(p(v_s)\leq3p_0/4\). Since \(v_s\) is the first
pair satisfying a stopping condition, its predecessor has parameter
greater than \(3p_0/4\), and \(v_s\) is its heavy child. Therefore
\(p(v_s)>9p_0/16\). It follows that
\(p(I_2,X_2)=p_0-p(v_s)<7p_0/16\). Thus
\(\mathcal F_1\) and \(\mathcal F_2\) have total parameter \(p_0\),
and each has parameter at most \(3p_0/4\). Retain both as recursive
systems.

Suppose instead that \(p(v_s)>3p_0/4\). Then
\(p(I_2,X_2)<p_0/4\). Since \(v_s\) satisfies a stopping condition,
it is terminal or balanced.

If \(v_s\) is terminal, then \(\mathcal F_1\) has weight zero and may
be omitted. Retain only \(\mathcal F_2\).

If \(v_s\) is nonterminal, then it is balanced. Apply
\eqref{eq:positive-recurrence} at \(v_s\) and replace
\(\mathcal F_1\) by its two children. Their total parameter is at
most \(p(v_s)\), and each has parameter at most
\(3p(v_s)/4\leq3p_0/4\). Together with \(\mathcal F_2\), they have
total parameter at most \(p_0\). The additional one-step error has
weight at most \(C_{k,t}\Gamma p(v_s)\Lambda\).

In every case, the recursive systems have total parameter at most
\(p_0\), and each has parameter at most \(3p_0/4\). The two
off-diagonal systems and, when necessary, the balanced-step error
have total weight at most \(C_{k,t}\Gamma p_0\Lambda\). This proves
both assertions.
\end{proof}

\begin{thm}\label{thm:positive-induction}
If \(\mathsf L_k(\Lambda)\) holds, then
\[
 \mathsf L_{k+1}\bigl(\Lambda\ell_B\bigr)
\]
holds. In particular, every system with \(t\)-fold intersections bounded by
\(B\) that is obtained by adjoining a \((k+1)\)-wise laminar family to a
base system in \(\mathcal C\) has weight at most
\[
 C_{k,t}\Gamma N_B\Lambda\ell_B.
\]
\end{thm}

\begin{proof}
Construct a recursion tree. At a balanced node
use~\eqref{eq:positive-recurrence}; at a node with a heavy child use
Proposition~\ref{prop:heavy-compression}. In either case, the recursive
children have total parameter at most that of the parent, each child has
parameter at most three quarters of that of the parent, and the total
error at the node is at most \(C_{k,t}\Gamma p(v)\Lambda\).
If \(S_d\) is the sum of the parameters at recursion depth \(d\), then
\(S_d\leq N_B\). Every positive parameter is an integer, and along a
branch it decreases by a factor of at least \(3/4\); hence the depth is
\(O(\ell_B)\). Summing the error estimates over all nodes proves the
theorem. Starting the construction at an arbitrary pair \((J,Y)\) gives
the bound required in the definition of \(\mathsf L_{k+1}\).
\end{proof}

\begin{cor}\label{cor:positive-fixed-k}
For every \(k\geq2\), \(\mathsf L_k(\ell_B^{\,k-1})\) holds.
Equivalently, adjoining one \(k\)-wise laminar family multiplies the
base estimate by at most \(C_{k,t}\ell_B^{k-1}\).
\end{cor}

\begin{proof}
The case \(k=2\) is Proposition~\ref{prop:base-extension}. Apply Theorem~\ref{thm:positive-induction} inductively.
\end{proof}

\section[Co-k-wise laminar families]
{Co-\texorpdfstring{\(k\)}{k}-wise laminar families}
\label{sect:negative-induction}

\begin{defn}
An indexed family \(\mathcal D=(D_i)_{i\in I}\) on \(X\) is \emph{co-\(k\)-wise laminar} if \((X\setminus D_i)_{i\in I}\) is \(k\)-wise laminar. Thus co-\(2\)-wise laminarity is co-laminarity.
\end{defn}

If \(Y\subseteq X\), then
\(Y\setminus(D_i\cap Y)=(X\setminus D_i)\cap Y\).
It follows that restriction to \(Y\) preserves co-\(k\)-wise laminarity
when \(Y\) is taken as the ground set.

Let \(\mathsf B_k(\Lambda)\) be the signed extension property obtained
from \(\mathsf L_k(\Lambda)\) by allowing the adjoined family to be either
\(k\)-wise laminar or co-\(k\)-wise laminar.
Proposition~\ref{prop:base-extension} gives \(\mathsf B_2(\ell_B)\).
The positive part of the induction is Theorem~\ref{thm:positive-induction}. We next prove the negative part.

\begin{prop}\label{prop:negative-induction}
If \(\mathsf B_k(\Lambda)\) holds, then adjoining a co-\((k+1)\)-wise laminar family satisfies the extension estimate
\[
 C_{k,t}\Gamma N_B\Lambda\ell_B.
\]
More precisely, if \(\mathcal F\) is the resulting system, then, for
every \(J\subseteq I\) and \(Y\subseteq X\),
\[
 w\bigl(\mathcal F\mathbin{|}(J,Y)\bigr)
 \leq C_{k,t}\Gamma\Delta_B(Y,J)\Lambda\ell_B.
\]
\end{prop}

\begin{proof}
Let \(\mathcal D=(D_i)_{i\in I}\) be co-\((k+1)\)-wise
laminar, let \(\mathcal F'=(F'_i)_{i\in I}\in\mathcal C\), and suppose
that \(F_i=F'_i\cap D_i\) defines a system
\(\mathcal F=(F_i)_{i\in I}\) whose \(t\)-fold intersections are
bounded by \(B\). Put \(K_i=X\setminus D_i\), so
\((K_i)_{i\in I}\) is \((k+1)\)-wise laminar. We prove the restricted
estimate directly. Fix \(J\subseteq I\) and \(Y\subseteq X\), and write
\(W(J,Y)=w(\mathcal F\mathbin{|}(J,Y))\) and
\(p(J,Y)=\Delta_B(Y,J)\).

Call \((J,Y)\) terminal if \(J=\varnothing\), \(Y=\varnothing\), or
\(K_i\cap Y=\varnothing\) for every \(i\in J\). In the first two
cases \(W(J,Y)=0\). In the third,
\(\mathcal F\mathbin{|}(J,Y)=\mathcal F'\mathbin{|}(J,Y)\);
this system belongs to \(\mathcal C\) and, being a restriction of
\(\mathcal F\), has \(t\)-fold intersections bounded by \(B\).
Hence \(W(J,Y)\leq\Gamma p(J,Y)\).

Suppose that \((J,Y)\) is nonterminal. Choose \(a\in J\) such that
\(U=K_a\cap Y\) is nonempty and maximal among the sets
\((K_i\cap Y)_{i\in J}\), and define
\[
\begin{aligned}
J_{\mathrm{in}}&=\{i\in J:\varnothing\neq K_i\cap Y\subsetneq U\},\\
J_{\mathrm{out}}&=\{i\in J:(K_i\cap Y)\cap U=\varnothing\},\\
J_{\mathrm{eq}}&=\{i\in J:K_i\cap Y=U\},\\
J_{\times}&=J\setminus
 (J_{\mathrm{in}}\cup J_{\mathrm{out}}\cup J_{\mathrm{eq}}).
\end{aligned}
\]
Thus \(J_{\mathrm{out}}\) includes the indices with empty complement
restriction. Apart from these empty restrictions, this is the same
four-class partition as in Section~\ref{sect:positive-induction}, with
\(D_i\) there replaced by \(K_i\). In particular, every member of
\(J_{\times}\) crosses \(U\). The structural children are
\((J_{\mathrm{in}},U)\) and
\((J_{\mathrm{out}}\cup J_{\times},Y\setminus U)\), and their total
parameter is at most \(p(J,Y)\).

The constraint is automatic on \(U\) for \(J_{\mathrm{out}}\), and on
\(Y\setminus U\) for \(J_{\mathrm{in}}\cup J_{\mathrm{eq}}\). Hence
the two systems
\[
 (F'_i\cap U)_{i\in J_{\mathrm{out}}},\qquad
 (F'_i\cap(Y\setminus U))_
   {i\in J_{\mathrm{in}}\cup J_{\mathrm{eq}}}
\]
belong to \(\mathcal C\), are subsystems of \(\mathcal F\), and have
total weight at most \(\Gamma p(J,Y)\). For
\(i\in J_{\mathrm{eq}}\cup J_{\times}\), put
\(E_i=K_i\cap U\). Lemma~\ref{lem:one-step-reduction}, applied to
\((K_i)\) on the indices with nonempty \(K_i\cap Y\), applies to exactly
the classes \(J_{\mathrm{eq}}\cup J_{\times}\). It shows that
\((E_i)\) is \(k\)-wise laminar on \(U\).
Therefore \((U\setminus E_i)\) is co-\(k\)-wise laminar, and
\[
 F'_i\cap(U\setminus E_i)=F_i\cap U.
\]
The corresponding system is a restriction of \(\mathcal F\), so
\(\mathsf B_k(\Lambda)\) bounds its weight by
\(C_{k,t}\Gamma p(J,Y)\Lambda\). Since \(\Lambda\geq1\), the preceding
systems give
\begin{equation}\label{eq:negative-recurrence}
\begin{aligned}
W(J,Y)\leq{}&
 W(J_{\mathrm{in}},U)
 +W(J_{\mathrm{out}}\cup J_{\times},Y\setminus U)\\
&+C_{k,t}\Gamma p(J,Y)\Lambda .
\end{aligned}
\end{equation}

Call a structural child heavy if its parameter is greater than
\(3p(J,Y)/4\), and call the pair balanced if neither child is heavy.
At most one child is heavy. Starting from a pair
\(v_0=(J_0,Y_0)\) with a heavy child, follow heavy children until the
first pair \(v_s=(J_s,Y_s)\) that is terminal, is balanced, or
satisfies \(p(v_s)\leq3p(v_0)/4\). This stopping time is finite because
the positive integer \(p\) decreases strictly along every nonempty
proper structural child.

Put \(I_1=J_s\), \(I_2=J_0\setminus J_s\),
\(X_1=Y_s\), and \(X_2=Y_0\setminus Y_s\). At the \(q\)-th step let
\(U_q\) be the pivot, let \(J_{\times}^q\) be the crossing class, and
let \(\mathcal O\) be the set of outside steps. For \(i\in I_1\), set
\[
 P_i=\bigcup_{\substack{q\in\mathcal O\\i\in J_{\times}^q}}
       (K_i\cap U_q).
\]
An index with empty \(K\)-restriction lies in the outside class; it
either persists only through outside steps or leaves at an inside step,
and consequently contributes an empty \(P_i\) or \(R_i\). We verify the
endpoint assertions directly, since such indices are retained in the
signed recursion. Put \(Z_q=Y_q\setminus Y_{q+1}\). The sets \(Z_q\)
are disjoint and have union \(X_2\). At an inside step, a surviving
index has \(K_i\cap Z_q=\varnothing\). At an outside step,
\(Z_q=U_q\); an index in \(J_{\mathrm{out}}^q\), including one with
empty restriction, misses \(U_q\), while an index in
\(J_{\times}^q\) contributes \(K_i\cap U_q\). Hence
\(P_i=K_i\cap X_2\).

If \(k\) nonempty \(P_i\) have a common point, that point lies in the
unique pivot region of one outside step, and all selected indices cross
that pivot. Lemma~\ref{lem:one-step-reduction}, with \(D_i\) replaced
by \(K_i\), gives a comparable selected pair. As in the last paragraph
of Lemma~\ref{lem:surviving-indices}, inclusion of the original fibers
implies inclusion of the corresponding \(P\)-sets. Thus the nonempty
\(P_i\) form a \(k\)-wise laminar family.

For \(i\in I_2\), put \(R_i=K_i\cap X_1\). An index with
\(R_i\neq\varnothing\) cannot leave at an outside step. If \(R_i=X_1\),
it is comparable with every other \(R\)-set. Otherwise it leaves at an
inside step as a crossing index. Taking the latest departure among \(k\)
such indices, the argument of Lemma~\ref{lem:deleted-indices} shows that
each selected \(K_i\) is incomparable with the latest pivot. The pivot is
distinct from the selected indices, and all these fibers have a common
point in \(X_1\). Their \((k+1)\)-wise laminarity therefore gives a
comparable selected pair. Hence the nonempty \(R_i\) form a
\(k\)-wise laminar family.

The two off-diagonal systems are therefore
\[
\begin{aligned}
\mathcal E_1
 &=\bigl(F'_i\cap(X_2\setminus P_i)\bigr)_{i\in I_1}
   =\mathcal F\mathbin{|}(I_1,X_2),\\
\mathcal E_2
 &=\bigl(F'_i\cap(X_1\setminus R_i)\bigr)_{i\in I_2}
   =\mathcal F\mathbin{|}(I_2,X_1).
\end{aligned}
\]
They are restrictions of \(\mathcal F\), their visible constraint
families are co-\(k\)-wise laminar, and
\[
 w(\mathcal E_1)+w(\mathcal E_2)
 \leq C_{k,t}\Gamma
 \bigl(p(I_1,X_2)+p(I_2,X_1)\bigr)\Lambda
 =C_{k,t}\Gamma p(J_0,Y_0)\Lambda.
\]
The four quadrants determined by
\(I_1\sqcup I_2=J_0\) and \(X_1\sqcup X_2=Y_0\) partition all
incidences. The diagonal systems are
\(\mathcal F\mathbin{|}(I_1,X_1)\) and
\(\mathcal F\mathbin{|}(I_2,X_2)\); their parameters sum to
\(p_0:=p(J_0,Y_0)\), and the first has parameter \(p(v_s)\).

The remaining compression is the numerical argument from the proof of
Proposition~\ref{prop:heavy-compression}. If
\(p(v_s)\leq3p_0/4\), minimality of \(s\) and heaviness give
\(p(v_s)>9p_0/16\), so the second diagonal has parameter less than
\(7p_0/16\); retain both diagonals. If \(p(v_s)>3p_0/4\), the second
diagonal has parameter less than \(p_0/4\). When \(v_s\) is terminal,
the first diagonal either has weight zero or belongs to \(\mathcal C\)
and has weight at most \(\Gamma p(v_s)\); include it among the error
terms. Otherwise \(v_s\) is balanced. Apply
\eqref{eq:negative-recurrence} at \(v_s\) and replace the first
diagonal by its two structural children. Their total parameter is at
most \(p(v_s)\), and each has parameter at most
\(3p(v_s)/4\leq3p_0/4\). Thus, exactly as in
Proposition~\ref{prop:heavy-compression}, a heavy node is replaced by
recursive systems of total parameter at most \(p_0\), each of parameter
at most \(3p_0/4\), together with error systems of total weight at most
\(C_{k,t}\Gamma p_0\Lambda\).

Finally, form the macro-recursion tree, using
\eqref{eq:negative-recurrence} at balanced nodes, the preceding
compression at heavy nodes, and the zero or base estimate at terminal
nodes. At every depth the total recursive parameter is at most
\(p(J,Y)\), while every nonterminal branch contracts by a factor of at
least \(3/4\). Its depth is therefore \(O(\ell_B)\), and summing the
error estimates gives
\[
 w\bigl(\mathcal F\mathbin{|}(J,Y)\bigr)
 \leq C_{k,t}\Gamma\Delta_B(Y,J)\Lambda\ell_B.
\]
Taking \(J=I\) and \(Y=X\) gives the global estimate.
\end{proof}

\begin{cor}\label{cor:signed-fixed-k}
For every \(k\geq2\),
\[
 \mathsf B_k\bigl(\ell_B^{\,k-1}\bigr)
\]
holds. Thus adjoining one \(k\)-wise laminar or co-\(k\)-wise laminar family multiplies the base estimate by at most \(C_{k,t}\ell_B^{k-1}\).
This assertion includes every induced restriction, with
\(\Delta_B(Y,J)\) in place of \(N_B\).
\end{cor}

\begin{proof}
The base case is Proposition~\ref{prop:base-extension}. If
\(\mathsf B_k(\ell_B^{k-1})\) holds,
Theorem~\ref{thm:positive-induction} gives the \((k+1)\)-wise laminar case
with factor \(\ell_B^k\), while Proposition~\ref{prop:negative-induction}
gives the co-\((k+1)\)-wise laminar case with the same factor.
\end{proof}

\section{A low-crossing sharpening}\label{sect:sharpening}

We continue with the notation and hypotheses of
Section~\ref{sect:setup}, including the class \(\mathcal C\), the constant
\(\Gamma\), and the base estimate~\eqref{eq:base-estimate}.

We first recall the set-system terminology used in this section. We
reserve \(X\) and \(\mathcal F\) for the indexed systems of
Section~\ref{sect:setup}; the auxiliary ground set below will instead be a
family of distinct fibers. Let \(U\) be a finite set and let
\(\mathcal R\subseteq2^U\) be a finite family. For \(S\subseteq U\), the
family \(\{R\cap S:R\in\mathcal R\}\) is the \emph{primal trace family}
on \(S\). The set \(S\) is \emph{shattered} if this family equals
\(2^S\), and the \emph{VC-dimension} of \((U,\mathcal R)\) is the
largest cardinality of a shattered set, with value \(0\) if no nonempty
set is shattered.

The dual notion restricts the ranges rather than the ground set. For a
positive integer \(q\), define
\[
 \pi_{\mathcal R}^*(q)
 =\max_{\substack{\mathcal Q\subseteq\mathcal R\\|\mathcal Q|\leq q}}
 \left|\left\{
 \{R\in\mathcal Q:u\in R\}:u\in U
 \right\}\right|.
\]
Thus \(\pi_{\mathcal R}^*(q)\) counts the point-membership patterns on a
subfamily of at most \(q\) ranges; it is not the cardinality of a primal
trace family. For \(q\leq|\mathcal R|\), this agrees with the usual
definition using exactly \(q\) ranges, since enlarging \(\mathcal Q\)
cannot identify two previously distinct patterns; for
\(q>|\mathcal R|\), it is constant. A range
\(R\in\mathcal R\) \emph{crosses} an edge of a path on \(U\) if it
contains exactly one endpoint. The \emph{crossing number} of the path is
the maximum, over \(R\in\mathcal R\), of the number of its edges crossed
by \(R\). Finally, write \(H_a=\sum_{j=1}^a1/j\) for the \(a\)-th
harmonic number.

An indexed family is \emph{simple} if distinct indices determine
distinct sets; we identify such a family with its ordinary collection of
members. The following estimate is proved directly by Tomon
\cite[Theorem~1.2]{TomonCrossFree}. It also follows from Knill's earlier
theorem on locally bounded width~\cite{KnillLocallyBoundedWidth}: for
\(k\geq3\), a \(k\)-wise laminar family is precisely a locally
\((k-1)\)-wide family, while the case \(k=2\) is the standard laminar
case.

\begin{thm}\label{thm:linear-size}
For every \(k\geq2\), there is a constant \(c_k\) such that every
simple \(k\)-wise laminar family \(\mathcal A\subseteq2^Y\) on a finite
ground set \(Y\) satisfies
\[
 |\mathcal A|\leq 1+c_k|Y|.
\]
\end{thm}

Both Knill's theorem and Tomon's formulation concern ordinary set
families, in which repeated copies of the same set are not counted. Indexed repetitions are handled separately in
Proposition~\ref{prop:low-contiguity}.

For an ordinary family \(\mathcal A\subseteq2^Y\), its \emph{dual set
system} has ground set \(\mathcal A\) and range family
\(\mathcal R_{\mathcal A}=\{R_x:x\in Y\}\), where
\(R_x=\{A\in\mathcal A:x\in A\}\); equal ranges are identified. Thus
dualization interchanges the roles of points and sets in the incidence
relation. For the simple family in Theorem~\ref{thm:linear-size},
we write \(\mathcal R=\mathcal R_{\mathcal A}\).

\begin{lem}\label{lem:dual-complexity}
Let \(k\geq2\), and let \(\mathcal A\subseteq2^Y\) be a simple
\(k\)-wise laminar family. Then \((\mathcal A,\mathcal R)\) has
VC-dimension at most \(k-1\). Moreover, for every \(Q\subseteq Y\),
\[
 \bigl|\{A\cap Q:A\in\mathcal A\}\bigr|
 \leq 1+c_k|Q|.
\]
Consequently, \(\pi_{\mathcal R}^*(q)\leq C_kq\) for every \(q\geq1\),
where \(C_k\) depends only on \(k\).
\end{lem}

\begin{proof}
Suppose that \(A_1,\ldots,A_k\in\mathcal A\) are shattered by
\(\mathcal R\). There is a point \(x_0\in Y\) such that
\(A_1,\ldots,A_k\in R_{x_0}\), so
\(x_0\in A_1\cap\cdots\cap A_k\). For each \(i\), there is also a
point \(x_i\in Y\) for which
\(R_{x_i}\cap\{A_1,\ldots,A_k\}=\{A_i\}\). Hence
\(x_i\in A_i\setminus A_j\) for every \(j\neq i\). In particular,
\(x_i\in A_i\setminus A_j\) and \(x_j\in A_j\setminus A_i\), so the
\(A_j\) are pairwise incomparable. This contradicts \(k\)-wise
laminarity and proves the VC-dimension bound.

After repeated traces are identified, the trace family
\(\{A\cap Q:A\in\mathcal A\}\) is \(k\)-wise laminar on \(Q\). Indeed,
if \(k\) distinct traces have a common point, choose one original member
of \(\mathcal A\) for each trace. These original members have a common
point, so two are comparable; their traces are then comparable as well.
Theorem~\ref{thm:linear-size} gives the displayed estimate.

Now let \(\mathcal Q\subseteq\mathcal R\) have cardinality at most
\(q\), and choose one point of \(Y\) representing each range in
\(\mathcal Q\). The membership patterns of the members of
\(\mathcal A\) on \(\mathcal Q\) are in bijection with their distinct
traces on this set of representative points. Thus their number is at most
\(1+c_kq\leq C_kq\), after increasing \(C_k\) if necessary.
\end{proof}

The following low-crossing path theorem is
\cite[Theorem~18]{BonnetDuronSylvesterZamaraev}, in a formulation based
on Welzl's method~\cite{Welzl1988}. We state the full result in the
notation above; the logarithm is to base \(2\).

\begin{thm}[Bonnet--Duron--Sylvester--Zamaraev]
\label{thm:bdsz-low-crossing}
Let \((U,\mathcal R)\) be a finite set system with
\(|U|=n\geq1\) and \(\mathcal R\neq\varnothing\), let
\(f\colon\mathbb R_{\geq0}\to\mathbb R_{\geq0}\) be strictly
increasing, and let \(d\geq1\) be an integer. Suppose that
\((U,\mathcal R)\) has VC-dimension at most \(d\) and
\(\pi_{\mathcal R}^*(m)\leq f(m)\) for every integer
\(1\leq m\leq n\). Then there is a spanning path on \(U\) with crossing
number at most
\[
 2\log_2|\mathcal R|
 +10d\sum_{j=1}^n\frac{1}{f^{-1}(j/2)}.
\]
\end{thm}

We use the theorem only with \(f(q)=Cq\). More precisely, suppose that
\(|U|=a\geq1\), \(|\mathcal R|=s\geq1\), the VC-dimension is at most
\(d\), and \(\pi_{\mathcal R}^*(q)\leq Cq\) for every \(q\geq1\),
where \(C\geq1\). Since \(f^{-1}(j/2)=j/(2C)\),
Theorem~\ref{thm:bdsz-low-crossing} gives a spanning path with crossing
number at most
\begin{equation}\label{eq:linear-dual-crossing}
 2\log_2s+20dCH_a.
\end{equation}
List the vertices in their order along this path. For a fixed range, its
indicator changes value exactly at the path edges crossed by the range.
If its set of ones consists of \(r\) maximal intervals, there are at
least \(2r-2\) such changes, since at most two of these intervals meet
the endpoints of the order. A crossing number of \(c\) therefore gives
at most \(1+c/2\) intervals. In particular, every member of
\(\mathcal R\) is a union of at most
\(1+\log_2s+10dCH_a\) intervals in this order.

\begin{prop}\label{prop:low-contiguity}
Let \(k\geq2\), and let \(\mathcal D=(D_i)_{i\in J}\) be \(k\)-wise
laminar on a finite set \(Y\). There is a linear order on \(J\) such
that, for every \(x\in Y\), each of
\[
 \{i\in J:x\in D_i\},\qquad \{i\in J:x\notin D_i\}
\]
is a union of at most
\(C_k(1+\log(1+|Y|+|J|))\) intervals.
\end{prop}

\begin{proof}
Let \(\mathcal A=\{D_i:i\in J\}\) be the family of distinct fibers. If
\(J=\varnothing\), \(Y=\varnothing\), or \(|\mathcal A|\leq1\), the
assertion is immediate. We may therefore assume that \(Y\neq\varnothing\)
and \(|\mathcal A|\geq2\). Passing to the distinct fibers preserves
\(k\)-wise laminarity.

For \(x\in Y\), put \(R_x=\{A\in\mathcal A:x\in A\}\), and let
\(\mathcal R=\{R_x:x\in Y\}\), with equal ranges identified. Write
\(a=|\mathcal A|\) and \(s=|\mathcal R|\), so \(a\leq|J|\) and
\(s\leq|Y|\). By Lemma~\ref{lem:dual-complexity},
\((\mathcal A,\mathcal R)\) has VC-dimension at most \(k-1\) and
\(\pi_{\mathcal R}^{*}(q)\leq C_kq\) for every \(q\geq1\). Under the
at-most-\(q\) convention, this estimate also holds when \(q>s\), since
the dual shatter function is then constant.

Apply the linear specialization~\eqref{eq:linear-dual-crossing} of
Theorem~\ref{thm:bdsz-low-crossing} with \(d=k-1\). It gives a linear
order on \(\mathcal A\) in which every \(R_x\) is a union of at most
\(1+\log_2s+10(k-1)C_kH_a\) intervals. Since \(s\leq|Y|\),
\(a\leq|J|\), and \(H_a\leq1+\log a\), this number is at most
\(C'_k(1+\log(1+|Y|+|J|))\), where \(C'_k\) depends only on \(k\).

Replace each \(A\in\mathcal A\), in this order, by the consecutive block
\(J_A=\{i\in J:D_i=A\}\), using an arbitrary order within each block.
For every \(x\in Y\), the set \(\{i\in J:x\in D_i\}\) is the union of
the blocks \(J_A\) for which \(A\in R_x\). Replacing a fiber by a
consecutive block creates no new transitions, so this set satisfies the
required interval bound. Its complement in \(J\) is
\(\{i\in J:x\notin D_i\}\). Since the complement of a union of \(q\)
intervals in a finite linear order is a union of at most \(q+1\)
intervals, the same estimate holds for this second set after increasing
the constant if necessary.
\end{proof}

\begin{thm}\label{thm:direct-signed-extension}
Let \(k\geq2\), let \(\mathcal F'=(F'_i)_{i\in I}\in\mathcal C\), let
\(\mathcal D=(D_i)_{i\in I}\) be \(k\)-wise laminar on \(X\), and
choose one sign. Put \(L_i=D_i\) in the positive case and
\(L_i=X\setminus D_i\) in the negative case. Suppose that
\(F_i=F'_i\cap L_i\) defines a system \(\mathcal F=(F_i)_{i\in I}\)
whose \(t\)-fold intersections are bounded by \(B\). Then
\[
 w(\mathcal F)\leq C_{k,t}\Gamma N_B\ell_B^2.
\]
For every \(J\subseteq I\) and \(Y\subseteq X\), one also has
\[
 w\bigl(\mathcal F\mathbin{|}(J,Y)\bigr)
 \leq C_{k,t}\Gamma\Delta_B(Y,J)\ell_B^2.
\]
\end{thm}

\begin{proof}
Fix \(J\subseteq I\) and \(Y\subseteq X\). The assertion is immediate
if \(J\) or \(Y\) is empty. The restricted family
\((D_i\cap Y)_{i\in J}\) remains \(k\)-wise laminar. Its positive
signed members are \(D_i\cap Y\), while its negative signed members are
\(Y\setminus(D_i\cap Y)\). Apply
Proposition~\ref{prop:low-contiguity}, and let \(A_x\subseteq J\) be the
set of indices for which the chosen signed condition holds at \(x\in Y\).
Each \(A_x\) is the union of at most
\(q=\lceil C_k\ell_B\rceil\) intervals in the resulting order on \(J\).

Use the balanced binary interval-tree construction from the proof of
Proposition~\ref{prop:base-extension}, now on the ordered index set
\(J\), and put \(h=1+\lceil\log_2|J|\rceil\). Decompose every maximal
interval of \(A_x\) into the maximal tree nodes that it contains. For
each fixed \(x\), the resulting node intervals are disjoint, cover
\(A_x\), and number at most \(2qh\). For a tree node \(u\), let \(J_u\)
be its index interval and let \(Y_u\) consist of the points whose
decomposition contains \(u\); omit nodes for which
\(Y_u=\varnothing\).

If \(x\in Y_u\) and \(i\in J_u\), then \(x\in L_i\), and therefore
\(F_i\cap Y_u=F'_i\cap Y_u\). Hence
\(\mathcal F\mathbin{|}(J_u,Y_u)
=\mathcal F'\mathbin{|}(J_u,Y_u)\). The right-hand side belongs to
\(\mathcal C\), and its equality with an induced restriction of the
final system shows that its \(t\)-fold intersections are bounded by
\(B\). Moreover, the blocks \(Y_u\times J_u\) assign each incidence of
\(\mathcal F\mathbin{|}(J,Y)\) exactly once. Indeed, if
\(x\in F_i\cap Y\), then \(i\in A_x\), and exactly one node in the
canonical decomposition of \(A_x\) contains \(i\). Conversely, the
pointwise equality above shows that every incidence counted in an
induced block is an incidence of the final system. Applying
\eqref{eq:base-estimate} to every block gives
\[
 w\bigl(\mathcal F\mathbin{|}(J,Y)\bigr)
 \leq\Gamma\sum_u\Delta_B(Y_u,J_u).
\]

For every \(x\), at most \(2qh\) nodes occur in its decomposition, and
at each tree depth the node intervals \(J_u\) are disjoint. Hence
\(\sum_u|Y_u|\leq2qh|Y|\) and
\(\sum_u|J_u|\leq h|J|\). Since
\(1+\log(1+|Y|+|J|)\leq\ell_B\), we have
\(q=O_k(\ell_B)\) and \(h=O(\ell_B)\). Using \(\ell_B\geq1\),
\[
 \sum_u\Delta_B(Y_u,J_u)
 \leq C_k\ell_B^2|Y|+C_k\ell_B(B+1)|J|
 \leq C_k\Delta_B(Y,J)\ell_B^2.
\]
This proves the restriction estimate, and \(J=I\), \(Y=X\) gives the
global estimate.
\end{proof}

Set \(\alpha_k=\min\{k-1,2\}\).

\begin{cor}\label{cor:direct-signed-fixed-k}
For every \(k\geq2\), \(\mathsf B_k(\ell_B^{\,\alpha_k})\) holds.
\end{cor}

\begin{proof}
For \(k=2\), positive families are laminar and negative families are
co-laminar, so Proposition~\ref{prop:base-extension} applies with one
factor \(\ell_B\). For \(k\geq3\), apply
Theorem~\ref{thm:direct-signed-extension}.
\end{proof}

\begin{rem}\label{rem:recursive-comparison}
The recursive argument of Sections~\ref{sect:positive-induction}
and~\ref{sect:negative-induction} gives the structural implication
\(\mathsf B_k(\Lambda)\Rightarrow
\mathsf B_{k+1}(\Lambda\ell_B)\). Theorem~\ref{thm:direct-signed-extension}
uses the additional ordering input to give the sharper exponent
\(\alpha_k\) when \(k\geq4\). Thus the two arguments provide distinct
forms of the signed extension estimate.
\end{rem}

\section{Signed families and Boolean combinations}\label{sect:signed}

The next proposition gives a linear estimate when the base system is
constant and only one signed family is adjoined.

\begin{prop}\label{prop:signed-base}
Let \(\mathcal D=(D_i)_{i\in I}\) be \(k\)-wise laminar, and define
\(\mathcal H=(H_i)_{i\in I}\) by either \(H_i=D_i\) for every \(i\),
or \(H_i=X\setminus D_i\) for every \(i\). If \(J\subseteq I\) and
\(Y\subseteq X\) are such that the \(t\)-fold intersections of
\(\mathcal H\mathbin{|}(J,Y)\) are bounded by \(B\), then
\[
 w\bigl(\mathcal H\mathbin{|}(J,Y)\bigr)
 \leq
 \begin{cases}
 (k-1)(t-1)|Y|+B|J|,&H_i=D_i,\\[2mm]
 (t-1)(k+t-2)|Y|+
 \binom{k+t-1}{t}B|J|,&H_i=X\setminus D_i.
 \end{cases}
\]
In particular, both quantities are at most
\(C_{k,t}\Delta_B(Y,J)\).
\end{prop}

\begin{proof}
Fix \(J\) and \(Y\) as in the statement. The family
\((D_i\cap Y)_{i\in J}\) remains \(k\)-wise laminar.

First suppose that \(H_i=D_i\). Let
\(J_{\leq B}=\{i\in J:|H_i\cap Y|\leq B\}\) and
\(J_{>B}=J\setminus J_{\leq B}\). The first set of indices contributes
at most \(B|J|\). Fix a linear order on \(J\). For each \(x\in Y\),
partially order
\[
 P_x=\{i\in J_{>B}:x\in H_i\}
\]
by inclusion of \(H_i\cap Y\), using the fixed order to compare indices
whose restricted fibers are equal. Every member of \(P_x\) contains
\(x\). Consequently, its width is at most \(k-1\), since an antichain
of cardinality \(k\) would contradict \(k\)-wise laminarity. Its
height is at most \(t-1\): a chain of \(t\) distinct
indices would have intersection equal to its smallest restricted fiber,
which has cardinality greater than \(B\). By Dilworth's theorem,
\(|P_x|\leq(k-1)(t-1)\). Double counting the incidences indexed by
\(J_{>B}\) proves
\[
 \sum_{i\in J}|H_i\cap Y|
 \leq (k-1)(t-1)|Y|+B|J|.
\]

Now suppose that \(H_i=X\setminus D_i\), and retain the notation
\(J_{\leq B}\) and \(J_{>B}\), defined using \(|H_i\cap Y|\). Order
\(J_{>B}\) by inclusion of \(D_i\cap Y\), again using the fixed order
within each class of equal restricted fibers. This poset has height at
most \(t-1\). Indeed, along a chain of \(t\) indices the sets
\(H_i\cap Y=Y\setminus(D_i\cap Y)\) form a reverse chain, and their
intersection is its smallest member, which has cardinality greater than
\(B\). By Mirsky's theorem, \(J_{>B}\) is the union of at most \(t-1\)
antichains.

Fix one such antichain \(A\), and write \(q=|A|\). For \(x\in Y\), put
\[
 d_x=|\{i\in A:x\in D_i\}|,
 \qquad h_x=|\{i\in A:x\in H_i\}|=q-d_x.
\]
The antichain \(A\) need not have bounded cardinality. Since the
restricted \(D_i\) indexed by \(A\) are pairwise incomparable,
\(k\)-wise laminarity gives the pointwise bound \(d_x\leq k-1\). If
\(q\leq k+t-2\), then
\[
 \sum_{i\in A}|H_i\cap Y|\leq(k+t-2)|Y|.
\]
Suppose that \(q\geq k+t-1\). Then \(h_x\geq q-k+1\geq t\) for every
\(x\in Y\), and double counting pairs consisting of a point and a
\(t\)-element subset of \(A\) gives
\[
 |Y|\binom{q-k+1}{t}
 \leq\sum_{x\in Y}\binom{h_x}{t}
 =\sum_{\substack{S\subseteq A\\|S|=t}}
 \left|\bigcap_{i\in S}(H_i\cap Y)\right|
 \leq B\binom qt.
\]
For \(q\geq k+t-1\), each factor in the following product decreases as
\(q\) increases, and hence
\[
 \frac{\binom qt}{\binom{q-k+1}{t}}
 =\prod_{j=0}^{t-1}\frac{q-j}{q-k+1-j}
 \leq\binom{k+t-1}{t}.
\]
Consequently, \(|Y|\leq\binom{k+t-1}{t}B\), and
\[
 \sum_{i\in A}|H_i\cap Y|
 \leq \binom{k+t-1}{t}Bq.
\]
There are at most \(t-1\) antichains. Summing the preceding estimates
and adding the contribution of \(J_{\leq B}\) gives
\[
 \sum_{i\in J}|H_i\cap Y|
 \leq (t-1)(k+t-2)|Y|
 +\binom{k+t-1}{t}B|J|.
\]
The proof includes empty and repeated fibers and remains valid when
\(B=0\).
\end{proof}

\begin{cor}\label{cor:one-signed-graph}
Let \(\phi(x;y)\) be a \((k,1)\)-semi-equation in \(M\), and let
\(V\subseteq M^{|x|}\) and \(W\subseteq M^{|y|}\) be finite. Define
\(E^+=\{(a,b)\in V\times W:M\models\phi(a;b)\}\) and
\(E^-=\{(a,b)\in V\times W:M\models\neg\phi(a;b)\}\). If the indicated
relation is \(K_{t,t}\)-free, then
\[
 \begin{aligned}
 |E^+|&\leq(k-1)(t-1)|V|+(t-1)|W|,\\
 |E^-|&\leq(t-1)(k+t-2)|V|
 +\binom{k+t-1}{t}(t-1)|W|,
 \end{aligned}
\]
respectively.
\end{cor}

\begin{proof}
Apply Proposition~\ref{prop:signed-base} with \(X=V\), \(I=W\), and
\(B=t-1\).
\end{proof}

\begin{lem}\label{lem:signed-conjunction}
Let \(m\geq0\), and let \(\mathcal C_0\) satisfy~\eqref{eq:base-estimate}. For \(1\leq q\leq m\), let
\[
 \mathcal D^{(q)}=(D_i^{(q)})_{i\in I}
\]
be \(k\)-wise laminar, and choose a sign \(\varepsilon_q\in\{+,-\}\) independent of \(i\). Set
\[
 L_i^{(q)}=
 \begin{cases}
 D_i^{(q)},&\varepsilon_q=+,\\
 X\setminus D_i^{(q)},&\varepsilon_q=-,
 \end{cases}
\qquad
 H_i=F'_i\cap\bigcap_{q=1}^mL_i^{(q)},
\]
where \(\mathcal F'\in\mathcal C_0\); when \(m=0\), the intersection of literals is understood to be \(X\). If the \(t\)-fold intersections of
\(\mathcal H=(H_i)_{i\in I}\) are bounded by \(B\), then
\[
 w(\mathcal H)
 \leq C_{k,t,m}\Gamma N_B\ell_B^{m\alpha_k}.
\]
For every \(J\subseteq I\) and \(Y\subseteq X\), one also has
\[
 w\bigl(\mathcal H\mathbin{|}(J,Y)\bigr)
 \leq C_{k,t,m}\Gamma\Delta_B(Y,J)\ell_B^{m\alpha_k}.
\]
\end{lem}

\begin{proof}
Proceed by induction on \(m\). For \(m=0\), the assertion
is~\eqref{eq:base-estimate}. After adjoining the first \(m-1\)
literals, the resulting class, including all its induced restrictions,
satisfies the base estimate with \(\Gamma\) replaced by
\[
 C_{k,t,m-1}\Gamma\ell_B^{(m-1)\alpha_k}.
\]
More precisely, take this class to consist of all systems obtained
from a member of \(\mathcal C_0\) by adjoining the first \(m-1\)
fixed signed families, together with all their induced restrictions.
It is closed under induced restrictions, and the induction hypothesis
bounds each member whose \(t\)-fold intersections are bounded by \(B\).
The intermediate system obtained from the first \(m-1\) literals need not
itself satisfy this intersection condition. This causes no additional
hypothesis: the base estimate for the new class is conditional on the
intersection condition, and every system to which that estimate is applied
in the signed extension is a subsystem of the final system. On a
restriction \((J,Y)\), a positive literal gives the \(k\)-wise laminar
family \((D_i^{(m)}\cap Y)_{i\in J}\), whereas a negative literal gives
the co-\(k\)-wise laminar family
\((Y\setminus(D_i^{(m)}\cap Y))_{i\in J}\). Therefore
Corollary~\ref{cor:direct-signed-fixed-k} applies on the ground set \(Y\) and
gives both the induction step and the stated restriction estimate.
\end{proof}

\begin{cor}\label{cor:constant-base-conjunction}
In the notation of Lemma~\ref{lem:signed-conjunction}, suppose that
\(m\geq1\) and
\(F'_i=X\) for every \(i\in I\). Then
\[
 w(\mathcal H)
 \leq C_{k,t,m}N_B\ell_B^{(m-1)\alpha_k}.
\]
For every \(J\subseteq I\) and \(Y\subseteq X\), one also has
\[
 w\bigl(\mathcal H\mathbin{|}(J,Y)\bigr)
 \leq C_{k,t,m}\Delta_B(Y,J)\ell_B^{(m-1)\alpha_k}.
\]
\end{cor}

\begin{proof}
Use the first signed literal as the base system.
Proposition~\ref{prop:signed-base} shows that the class consisting of all
such systems and their induced restrictions satisfies
\eqref{eq:base-estimate} with a constant depending only on \(k\) and
\(t\). Apply Lemma~\ref{lem:signed-conjunction} to the remaining
\(m-1\) signed families. This also covers \(m=1\), when no further
family is adjoined.
\end{proof}

\begin{thm}\label{thm:boolean-combinations}
Let \(m\geq1\) and \(\Psi\colon\{0,1\}^m\to\{0,1\}\). Let
\(\mathcal F'\in\mathcal C\), and for \(1\leq q\leq m\) let
\(\mathcal D^{(q)}=(D_i^{(q)})_{i\in I}\) be \(k\)-wise laminar. Define
\[
 H_i=F'_i\cap
 \left\{x\in X:
 \Psi\bigl(\ind_{x\in D_i^{(1)}},\ldots,\ind_{x\in D_i^{(m)}}\bigr)=1
 \right\}.
\]
If the \(t\)-fold intersections of \(\mathcal H=(H_i)_{i\in I}\) are
bounded by \(B\), then
\[
 w(\mathcal H)
 \leq C_{k,t,m}\Gamma N_B\ell_B^{m\alpha_k}.
\]
For every \(J\subseteq I\) and \(Y\subseteq X\), one also has
\[
 w\bigl(\mathcal H\mathbin{|}(J,Y)\bigr)
 \leq C_{k,t,m}\Gamma\Delta_B(Y,J)\ell_B^{m\alpha_k}.
\]
If \(F'_i=X\) for every \(i\in I\), then the sharper estimates
\[
 w(\mathcal H)
 \leq C_{k,t,m}N_B\ell_B^{(m-1)\alpha_k}
\]
and
\[
 w\bigl(\mathcal H\mathbin{|}(J,Y)\bigr)
 \leq C_{k,t,m}\Delta_B(Y,J)\ell_B^{(m-1)\alpha_k}
\]
hold.
\end{thm}

\begin{proof}
Let
\[
 \mathcal A=\{a\in\{0,1\}^m:\Psi(a)=1\}.
\]
For \(a=(a_1,\ldots,a_m)\in\mathcal A\), set
\[
 L_{i,a}^{(q)}=
 \begin{cases}
 D_i^{(q)},&a_q=1,\\
 X\setminus D_i^{(q)},&a_q=0,
 \end{cases}
\qquad
 H_{i,a}=F'_i\cap\bigcap_{q=1}^mL_{i,a}^{(q)}.
\]
For each \(i\), the membership-pattern cells partition \(X\), so
\[
 H_i=\bigsqcup_{a\in\mathcal A}H_{i,a},
 \qquad
 w(\mathcal H)=\sum_{a\in\mathcal A}w(\mathcal H_a).
\]
Every \(\mathcal H_a\) is a subsystem of \(\mathcal H\), so its
\(t\)-fold intersections are bounded by \(B\).
Lemma~\ref{lem:signed-conjunction} applies to each \(a\), and
\(|\mathcal A|\leq2^m\). Summing gives the general estimates. If the
base system is constant, apply
Corollary~\ref{cor:constant-base-conjunction} to each cell instead.
For a restriction \((J,Y)\), the corresponding cells again partition
each fiber, and every accepted cell is a subsystem of
\(\mathcal H\mathbin{|}(J,Y)\). Applying the same argument on the ground
set \(Y\) gives the restriction estimates.
\end{proof}

\section{Proof of Theorem \ref{thm:absolute-main}}\label{sect:higher-arity}

\begin{proof}[Proof of Theorem~\ref{thm:absolute-main}]
We argue by induction on \(r\), proving the assertion simultaneously
for every fixed formula of arity \(r\). Suppose first that \(r=2\).
By the definition of a \(1\)-semi-equational theory, with respect to the
partition \(x_1;(x_2,z)\), the formula \(\varphi(x_1,x_2;z)\) is
equivalent modulo \(T\) to a Boolean combination of fixed formulas
\(\psi_1,\ldots,\psi_m\), given by a fixed Boolean function \(\Psi\),
where \(\psi_q\) is a \((k_q,1)\)-semi-equation. If the Boolean
combination has no literals, adjoin the tautological
\((2,1)\)-semi-equation and let \(\Psi\) ignore it. Thus we may assume
\(m\geq1\). Put \(k=\max_qk_q\). Each \((k_q,1)\)-semi-equation is also a
\((k,1)\)-semi-equation: from \(k\) fibers with a common point, apply its
defining condition to any \(k_q\) of them. After fixing \(d\), set
\(X=V_1\) and \(I=V_2\). For \(b\in I\), define
\(D_b^{(q)}=\{a\in X:M\models\psi_q(a;b,d)\}\) and
\(H_b=\{a\in X:M\models\varphi(a,b;d)\}\).
Each \((D_b^{(q)})_{b\in I}\) is \(k\)-wise laminar. Since \(E\) is
\(K_{t,t}\)-free, the intersection of any \(t\) fibers \(H_b\) with
distinct indices has cardinality at most \(t-1\). Indeed, \(t\) points
in such an intersection together with the \(t\) indices would form a
copy of \(K_{t,t}\).

Apply the constant-base conclusion of
Theorem~\ref{thm:boolean-combinations} with \(B=t-1\). Here
\(N_B=|V_1|+t|V_2|\leq tn\), and
\(\ell_B=O_t(1+\log(1+n))\). Hence
\[
 |E|=O_{T,\varphi,t,2}\!\left(
 n(1+\log(1+n))^{(m-1)\alpha_k}\right),
\]
where \(m\) and \(k\) depend only on \(T\) and \(\varphi\). This proves
the assertion for \(r=2\), uniformly in \(d\).

Assume \(r\geq3\) and that the result holds in arity \(r-1\). Put
\(X=V_1\times\cdots\times V_{r-1}\) and \(I=V_r\), and, for \(b\in I\),
let \(H_b=\{a\in X:M\models\varphi(a,b;d)\}\). With respect to the
partition \((x_1,\ldots,x_{r-1});(x_r,z)\), the formula \(\varphi\) is
equivalent modulo \(T\) to a Boolean combination of fixed formulas
\(\psi_1,\ldots,\psi_m\), given by a fixed Boolean function, where
\(\psi_q\) is a \((k_q,1)\)-semi-equation. As above, a tautological
literal may be adjoined, so assume \(m\geq1\). Put \(k=\max_q k_q\).
After fixing \(d\), the
fibers of every \(\psi_q\) over \(b\in I\) form a \(k\)-wise laminar
family on \(X\).

For distinct \(b_1,\ldots,b_t\in I\), set
\(Z_{b_1,\ldots,b_t}=H_{b_1}\cap\cdots\cap H_{b_t}\). This relation is
\(K_{t,\ldots,t}\)-free on the first \(r-1\) parts: otherwise a forbidden
product in \(Z_{b_1,\ldots,b_t}\), together with
\(\{b_1,\ldots,b_t\}\), would give a \(K_{t,\ldots,t}\) in \(E\).
Moreover, these intersections are uniformly defined by the fixed formula
\[
 \theta(x_1,\ldots,x_{r-1};u_1,\ldots,u_t,z)
 =\bigwedge_{j=1}^t
 \varphi(x_1,\ldots,x_{r-1},u_j;z).
\]
The total size of the first \(r-1\) vertex parts is at most \(n\).
The induction hypothesis applied to \(\theta\) gives constants
\(C_0,a\geq0\), independent of \(b_1,\ldots,b_t,d\), such that
\[
 |Z_{b_1,\ldots,b_t}|
 \leq C_0n^{r-2}\bigl(1+\log(1+n)\bigr)^a.
\]
Let \(B\) be the ceiling of the right-hand side. Thus
\((H_b)_{b\in I}\) has \(t\)-fold intersections bounded by the same \(B\)
for every choice of distinct indices. This is the only use of definability
of the fiber intersections; the combinatorial estimates of
Sections~\ref{sect:setup}--\ref{sect:signed} require no
definable order on \(X\) or \(I\).

Applying the constant-base conclusion of
Theorem~\ref{thm:boolean-combinations} to the families defined by the
\(\psi_q\) gives
\[
 |E|\leq C_{k,t,m}\bigl(|X|+(B+1)|I|\bigr)
 \left(1+\log\bigl(1+|X|+(B+1)|I|\bigr)\right)^{(m-1)\alpha_k}.
\]
Since \(|X|=\prod_{j=1}^{r-1}|V_j|\leq n^{r-1}\),
\(|I|=|V_r|\leq n\), and
\(B=O_{T,\varphi,t,r}\bigl(n^{r-2}(1+\log(1+n))^a\bigr)\), for
\(n\geq2\) one has
\[
 |X|+(B+1)|I|
 =O_{T,\varphi,t,r}\!\left(
 n^{r-1}(1+\log(1+n))^a\right),
\]
and the remaining logarithmic factor is
\(O_{T,\varphi,t,r}(1+\log(1+n))\). The asserted estimate follows, for
example with \(c=a+(m-1)\alpha_k\). If \(n\leq1\), at least one vertex
part is empty, so \(E=\varnothing\).
\end{proof}

\begingroup
\sloppy
\printbibliography
\endgroup

\end{document}